\documentclass[11pt,A4paper]{amsart}
\usepackage{verbatim, amscd, epsfig,amsmath,amsthm,amstext,amssymb,hyperref,tikz
}
\usepackage{amssymb, booktabs, longtable}
\usepackage[margin=3.8cm]{geometry}

\theoremstyle{plain}
\newtheorem*{theorem*}{Theorem}
\newtheorem{theorem}{Theorem}[section]

\newtheorem{cor}[theorem]{Corollary}
\newtheorem{corollary}[theorem]{Corollary}

\theoremstyle{definition}

\newtheorem{rem}[theorem]{Remark}

\newcommand{\R}{\mathbb{ R}}
\newcommand{\C}{\mathbb{ C}}
\newcommand{\Z}{\mathbb{ Z}}

\DeclareMathOperator{\tr}{tr}

\newcommand{\g}{\mathfrak{g}}

\renewcommand{\t}{\mathfrak{t}}
\newcommand{\mf}{\mathfrak}
\newcommand{\bz}{{\bar{z}}}

\newcommand{\Ad}{\operatorname{Ad}}
\newcommand{\ad}{\operatorname{ad}}

\newcommand{\diag}{\operatorname{diag}}
\newcommand{\p}{\mathfrak{p}}
\newcommand{\Tor}{\C/\Lambda}

\newcommand{\Fl}{\mathrm{Fl}}

\begin{document}
\title
{Harmonic tori in De Sitter spaces $S^{2n}_1$ }
\author{Emma Carberry, Katharine Turner}
\date{\today}
\maketitle

\begin{abstract}
We show that all
 superconformal harmonic immersions from  genus one surfaces into   de Sitter spaces $ S ^ {2n}_1 $  with globally defined harmonic sequence are of finite-type and hence result merely from solving a pair of ordinary differential equations. As an application, we prove that all Willmore tori in $ S ^ 3 $ without umbilic points can be constructed in this simple way.
\end{abstract}

\section{Introduction}\label{introduction}
 De Sitter spaces $ S ^ m_1 $ are the unit spheres in Minkowski space $\R ^ {m, 1} $ and are a fundamental example of a non-compact pseudo-Riemannian symmetric space. 
We study harmonic maps of the complex plane into de Sitter spaces $ S ^ {2n}_1 $ which have the additional pleasant property of being superconformal. This is a natural orthogonality property of derivatives of the map and  characterises those harmonic maps which correspond to solutions of the $ SO (2n, 1) $ affine Toda field equations. We prove that all superconformal harmonic maps from a genus one surface into $ S ^ {2n}_1 $ whose harmonic sequence is everywhere defined are of finite type. This means that they result from integrating a pair of commuting vector fields on a finite dimensional Lie algebra and hence can be constructed far more easily than solving the Euler-Lagrange equations.
 Willmore tori in $ S ^ 3 $ without umbilic points may be studied via their conformal Gauss map, which is a harmonic map into $ S ^ 4_1 $ and we obtain that all such Willmore tori are finite type.

Harmonic immersions of surfaces into {\em compact} Lie groups and symmetric spaces has been well-studied using integrable systems techniques. In particular, a series of papers addressed the question of whether harmonic immersions of genus one surfaces into various compact symmetric spaces are all of finite type  \cite {Hitchin:90, PS:89, Bobenko:91, BPW:95, FPPS:92}, culminating in the general results of \cite{BFPP:93} giving the mild additional assumptions necessary for harmonic maps of  tori   into a compact symmetric space to be of finite type. In the recent work  \cite{CT:11} we considered harmonic immersions into homogeneous spaces $ G/T $ where $ G $ is a simple Lie group (not necessarily compact) and $ T $ a Cartan subgroup. We showed that an immersion of a 2-torus into $ G/T $ which possesses a Toda frame is necessarily of finite type and characterised those $ k $-symmetric spaces $ G/T $ to which this theory applies in terms of extended Dynkin diagrams. In this manuscript we show that superconformal immersions of a surface into $ S ^ {2n}_1 $ whose harmonic sequence is everywhere defined possess a lift into $ SO (2n, 1)/T $ which has a Toda frame and hence the theory of \cite{CT:11} can be applied. This work is both the analogous to and a generalisation of the methods of \cite{BPW:95}, in which the compact case was addressed including the study of superconformal maps into Euclidean spheres.

De Sitter spaces are a particularly natural example to study as they are of interest both in mathematics and in  physics. For example, harmonic maps of Euclidean surface into de Sitter spaces arise naturally in solid-state physics, where they describe localisation properties in disordered conductors and superconductors  \cite {AP:88, Efetov:83, Oppermann:87}.
This finite type result is a crucial step in providing a spectral curve construction for these harmonic tori, such as have proven most helpful in the study of harmonic maps of tori into various compact symmetric spaces  \cite {Hitchin:90,PS:89, McIntosh:95, McIntosh:96, KSS:10, Haskins:04}. 

Harmonic maps of surfaces into spheres (Euclidean or Lorentzian) may be classified by their isotropy order. This is the dimension of the maximal isotropic subspace of the complexified tangent space spanned by successive $ z $-derivatives of the map. The isotropic maps are those with maximal isotropy order. These are all obtained from projections of holomorphic horizontal maps into the appropriate twistor space
\cite {Calabi:67, Bryant:82, Bryant:85, Bryant:84, Ejiri:88},
and so we concentrate on the study of the remaining, non-isotropic, harmonic maps. 
\emph{Superconformal} harmonic maps are those with the penultimate isotropy order.
The isotropy order of a harmonic map $ f $ from a surface into $ S ^ {2n}_1 $ can be measured by the length of its harmonic  sequence, which is obtained by applying the Gram-Schmidt procedure to successive $ z $-derivatives of  $ f $. At points  where an element in the harmonic sequence is nulllike but non-vanishing, the sequence cannot be further defined. Fortunately  \cite {Hulett:05} 
this can only occur at isolated points. We prove that a harmonic map $ f $ has a cyclic primitive lift to the full isotropic flag manifold $ SO (2n +1)/T $ if and only if it is superconformal and its harmonic sequence is defined everywhere. We show that superconformal harmonic maps with everywhere defined harmonic sequence precisely correspond to solutions of the affine Toda field equation for $ SO (2n , 1) $. This is  in natural analogy with the Euclidean case, although with the twist that we compute with respect to a general cyclic element of the Lie algebra, so our correspondence is with solutions to a natural generalisation of the classical affine Toda equations.  Using this correspondence, we apply the techniques of our previous work \cite {CT:11} to prove that all superconformal genus one surfaces  in $ S ^ {2n}_1 $ whose harmonic sequence is everywhere defined are of finite type.

We end with an application of our results to  Willmore tori in $ S ^ 3 $. An immersed surface in $\R ^ 3 $ is  \emph{Willmore} when it is critical for the Willmore functional $\int H ^ 2 dA $, where $ H $ denotes the mean curvature. Since this functional is conformally invariant, it is equivalent to study Willmore surfaces in $ S ^ 3 $. The famous Willmore conjecture \cite {Willmore:65} proposes that for genus one surfaces the minimum value of the Willmore functional is $ 2\pi ^ 2 $ and is achieved only for the Clifford torus. The conformal Gauss map of a surface in $ S ^ 3 $ is defined away from umbillic points and is a map into $ S ^ 4_1 $. The conformal Gauss map is harmonic precisely when the original immersion is Willmore. Hence our results apply to give a construction of all Willmore tori in $ S ^ 3 $ without umbillic points from a pair of ordinary differential equations on a finite dimensional Lie algebra. That such surfaces are of finite type was  shown in \cite{Bohle:08} without the exclusion of umbilic points, using methods and quaternionic holomorphic geometry, but for this case our methods give a much simpler proof.

Our results raise a number of interesting questions for further research. This analysis certainly motivates the question of finding conditions under which the harmonic sequence is everywhere defined. It also naturally prompts generalisation. We have shown what modifications of the Euclidean methods were needed to tackle the case of superconformal harmonic surfaces in $ S ^ {2n}_1 $ and more generally, the affine Toda equations for a  Lie group which is not necessarily compact. One would expect that  harmonic maps into $ S ^ {2n}_1 $ with lower isotropy orders whose harmonic sequences are everywhere defined have primitive lifts into partial flag manifolds and that using this lift the genus one such surfaces can be shown to be of finite type. This would be the natural extension of the results of Burstall  \cite {Burstall:95} in the Euclidean situation. In another direction, this finite type result provides a crucial ingredient for the construction of a spectral curve correspondence for superconformal harmonic tori in de Sitter spaces, as has been achieved for harmonic genus one surfaces in a number of compact target spaces \cite {Hitchin:90, PS:89, FPPS:92, McIntosh:95, McIntosh:96, McIntosh:02, MR:10}. 
The importance of a spectral curve correspondence is that it gives a purely algebro-geometric description of the harmonic map and hence yields appropriate tools for studying the moduli space along with a natural integer invariant, the spectral genus. Spectral curve correspondences have been used in \cite {EKT:93, Jaggy:94, Carberry:04, CM:03, CS:12} 
 to construct families of harmonic maps of arbitrarily high dimension. For superconformal harmonic tori in $ S ^ {2n}_1 $, the case $ n = 2 $ is of particular interest since it corresponds to Willmore tori in $ S ^ 3 $. A generalisation of the spectral curve constructed using integrable systems has been utilised to achieve exciting progress towards this conjecture \cite{Schmidt:02}. One would expect that as in the constant mean curvature case  \cite {CLP:09} 
the integrable systems spectral curve should be a partial normalisation of the one utilised in  \cite {Schmidt:02, Bohle:08} and should be able to be interpreted as a  subspace of the  space of Darboux transforms arising from the loop of flat connections. 

In section 2 we consider harmonic immersions of the plane into de Sitter spheres $ S ^ {2n}_1 $ and show that they have a primitive lift into the full isotropic flag manifold $ SO (2n, 1)/T $ precisely when they are superconformal and have an everywhere defined harmonic sequence. In section 3 we use this result and the results of \cite {CT:11} to prove that superconformal harmonic maps of tori with globally defined harmonic sequences are finite type. Section 4 contains the application of these results to Willmore surfaces.

\section{Superconformal maps of the complex plane to de Sitter spaces $ S ^ {2n}_1 $}\label{primitive}
%

We will consider harmonic maps from the complex plane into de Sitter spaces $ S ^ {2n}_1 $ and derive necessary and sufficient conditions for such a map to have a cyclic primitive lift into the isotropic flag manifold $ SO (2n, 1)/T $. Such maps correspond to solutions of the affine Toda equations for $\mathfrak {so} (2n, 1) $. 
We will later show that all doubly periodic cyclic primitive maps of the plane into $ SO (2n, 1)/T $ are of finite type which allows us to recover the corresponding harmonic map by integrating a pair of commuting vector fields on a finite dimensional loop algebra.


Let $\R ^ {2n, 1} $ denote $\R ^ {2n +1} $ with the Minkowski inner product
\[
 x ^ 1y ^ 1+ x ^ 2y ^ 2+\cdots + x ^ {2n}y ^ {2n} - x ^ {2n +1} y ^ {2n +1} = x ^ t\upsilon y ,
\] where $\upsilon =\diag (1,\ldots, 1, -1) $. 
The unit sphere $ S ^ {2n}_1\subset\R^ {{2n}, 1} $ with respect to the Minkowski metric is called ($ {2n} $-dimensional) de Sitter space and in particular the case $ n = 2 $ plays an important role in general relativity. De Sitter  space is acted upon transitively by the de Sitter group
$ G = SO ({2n}, 1) $ of orientation preserving isometries of $\R ^ {{2n}, 1} $, and the identity component of the stabiliser of $ (1, 0,\ldots , 0) $ is 
$ H =\diag\left (1, SO ({2n} -1, 1)\right) $. Thus we may view $ S ^ {2n}_1 $ as a homogeneous space $ SO ({2n}, 1)/H $ and it is furthermore a symmetric space with respect to the involution $\tau$ 
of $ SO ({2n}, 1) $ consisting of conjugation by $ h =\diag (1, -1,\ldots , -1) $. Explicitly, $ SO (2n, 1) $ is the subgroup of $ GL (2n +1,\R) $ given by the condition $ g ^ t\upsilon g =\upsilon $ and we observe that the Lorentzian metric on $ S ^ {2n}_1 $ inherited from $\R ^ {2n, 1} $ is given by the scalar multiple $ \frac 12\tr (\ad_X\ad_Y) $ of the Killing form on $ SO (2n, 1) $. 

The  \emph{isotropy order} of a harmonic map $f $ of a surface $ M ^ 2$ into $ S ^{2n}_1 $ is the integer $ r\geq 0 $ such that
\begin{align*}
\langle \partial_z^{a} f, \partial_z^{b} f\rangle &\equiv 0, \qquad\qquad \text{for } 1\leq a+ b \leq 2r +1, a,b\geq 0\\
\langle \partial_z^{r +1} f, \partial_z^{r +1} f\rangle &\not\equiv 0.
\end{align*}
 Here $\langle\cdot  ,\cdot \rangle$ denotes the complex bilinear product on $\C^{2n+1}_1$ and we will write $\|v\|^2 = \langle v, \bar{v}\rangle$.

A harmonic map having the maximal isotropy order $ r = n $ is variously termed  isotropic, superminimal or pseudo-holomorphic. Isotropic surfaces in $ S ^ {2n}_1 $ include all harmonic maps of $ S ^ 2 $, and can be expressed holomorphically in terms of a Weierstrass-type representation \cite {Calabi:67, Bryant:82, Bryant:85, Bryant:84, Ejiri:88}.  A harmonic map of isotropy order $ r\geq 1 $ is weakly conformal and hence space-like. A \emph{superconformal}  $f:\C\rightarrow S ^ {2n}_1 $ is one with the penultimate isotropy order $ r = n -1 $ and as
we shall see, superconformal maps have an interesting connection with the affine Toda field equations for $ SO (2n, 1) $. For $ n = 1 $ as well as the physically interesting case $ n = 2 $, harmonic maps into $ S ^ {2n}_1 $ must be either isotropic or superconformal.

Applying the Gram-Schmidt orthogonalisation process to successive $z$-derivatives of $f$, we inductively define the harmonic sequence $\{f_0,f_1, \ldots, f_{r}\}$ of a non-constant harmonic map $ f:M ^ 2\rightarrow S ^ {2n}_1 $  by
\begin{equation}\label{harm}
f_0 = f, \qquad
f_{j+1} = \partial_z f_j - \frac{\langle \partial_z f_j, \overline{f_j} \rangle}{\|f_j\|^2}f_j\mbox{ wherever $\|f_j\|^2\neq 0$}
\end{equation}
 and extend by continuity wherever $f_j=0$.
This sequence for maps into Lorentzian space has been studied in detail in  \cite{Hulett:05}, to which we refer the reader for further exposition (note that $\langle\cdot  ,\cdot  \rangle$ in  \cite{Hulett:05} denotes the Hermitian product rather than the complex bilinear product). It was shown in \cite [Lemma 3.1] {Hulett:05}  that the $f_j$ are defined on an open dense subset of $\C$,   satisfy
\begin{align*}
\partial_{\bar {z}} f_{j+1} = -\frac{\|f_{j+1}\|^2}{\|f_j \|^2} f_j & \quad \text{ for } 0\leq j \leq r \\
\langle f_j, \overline{f_k}\rangle = 0 & \quad \text{ unless } j=k
\end{align*}
and that the zeros of the $ f_j $ are isolated  whenever $ f_j $ does not vanish identically.
Note that for $0\leq j,k \leq r$ we have $ \langle f_j, \partial_z f_k\rangle =0$ and also
\begin {align*}
\langle  f_j,f_k\rangle & =0 \quad\text { unless } \quad0=j=k,\\
 \langle \bar{f_j}, \partial_z \bar{f_k}\rangle & =0 \quad\text { unless } \quad0=j=k-1.
\end {align*}


In Theorem ~\ref{thm:cyclicprimitive} we will show that a superconformal harmonic map whose harmonic sequence is everywhere defined possesses a cyclic primitive lift into the isotropic flag manifold
\[
\Fl (S ^ {2n}_1) =\{V_1\subset V_2\subset\cdots \subset V_{n -1}\subset T ^\C S ^ {2n}_1\mid V_j\mbox { is an isotropic sub-bundle of dimension $ j $}\}.
\]
We say that a sub-bundle is isotropic if its fibres are each isotropic as subspaces of $\C ^ {2n}_1 $.


Denote by $ T $ the Cartan subgroup $\diag \{1, SO(2), SO(2), \ldots SO(1,1))\}$ of $ SO ({2n}, 1)$ and write $\t $ for the corresponding Lie algebra and $\t ^\C $ for its complexification. 
Geometrically, the homogeneous space $ SO (2n, 1)/T $ is the full isotropic flag manifold
\[
\Fl (S ^ {2n}_1)
 =\{V_1\subset V_2\subset\cdots \subset V_{n -1}\subset T ^\C S ^ {2n}_1\mid V_j\mbox { is an isotropic sub-bundle of dimension $ j $}\}
\]
where we say that a sub-bundle is isotropic if its fibres are each isotropic with respect to the complex bilinear form $\langle\cdot,\cdot \rangle $.

 To define what it means for a map into $ SO (2n, 1)/T\cong \Fl (S ^ {2n}_1) $ to be cyclic primitive we shall exhibit the isotropic flag manifold as an $ n $-symmetric space and it will be advantageous to express this structure in terms of the standard choice of simple roots for $\mathfrak {so} (2n, 1,\C) $, which we now recall.

Define $ \tilde {a}_k\in\mathfrak t ^*$, $ k = 1, \ldots , n $ by
\[
\tilde {a}_k\left (\diag\left \{0,
\left (
\begin {array} {cc}
0 & a_1\\
- a_1 & 0\end {array}\right),
 \ldots
\left (
\begin {array} {cc}
0 & a_n\\
a_n & 0\end {array}\right)\right\}\right) =a_k.
\]
We choose as simple roots of $ \mathfrak {so} (2n, 1,\C) $ the roots $\alpha_1 = i\tilde {a}_1$, $\alpha_k = i\tilde {a}_k-i\tilde {a}_{k-1}$ for $1< k < n$, and $\alpha_n= \tilde {a}_n - i\tilde {a}_{n-1}$\label {page:simple}. The lowest root is then
\[
 \alpha_0: =- 2\alpha_1 - 2 \alpha_2 -\ldots -2\alpha_{n-1} - \alpha_n  = -\tilde {a}_n - i\tilde {a}_{n-1}, 
\]
which is of height $-2n+1$.

For ease of notation, when $ j<k$ we define
\[
[j,k] :=
\begin{cases}
 E_{kj} - E_{jk} & \text{ if } k\neq 2n+1\\
 E_{kj} + E_{jk} & \text{ if } k=2n+1.
\end{cases}
\]
where  $E_{jk}$ is the $(2n+1)\times(2n+1)$ elementary matrix with $1$ in the $(j,k)$ entry and zeroes elsewhere.
Then the corresponding dual basis of $ (\mathfrak{t} ^\C) ^*$ with respect to the Killing form is
\begin{equation}\label {eq:dual}
\eta_j= i\sum_{l=j}^{n-1} [2l,2l+1]  - [2n, 2n+1]\;\text {for $ 1\leq j\leq n -1 $ and }
\eta_n= - [2n,2n+1].
\end{equation}


A homogeneous space $ G/K $ is a {\it $ k $-symmetric space} ($ k >1 $) if there is an automorphism $\sigma: G\rightarrow G $ of order $ k $ such that
\[
(G ^\sigma)_0\subset K\subset G ^\sigma
\] where $ G ^\sigma $ denotes the fixed point set of $\sigma $, and $ (G ^\sigma)_0 $ the identity component of $ G ^\sigma $. When $ k = 2 $, we recover the notion of a symmetric space.

Consider the automorphism $\sigma $ of $ SO (2n, 1,\C) $ consisting of conjugation by
\[
 \exp \Bigl( \frac {\pi i } {n}\sum_{j = 1} ^ {n} \eta_j\Bigr) 
=\diag\left (1, R_{\left (\frac {\pi} {n}\right)}, R_{\left (\frac {2\pi} {n}\right)},\ldots, R_{\left (\frac {(n -1)\pi} {n}\right)}, - I_ {2 }\right),
\]
where $ R_\theta $ denotes rotation of the plane anticlockwise through angle $\theta $.

Denote also by $\sigma $ the corresponding Lie algebra automorphism
\[
\Ad_{ \exp ( \frac {\pi i } {n}\sum_{j = 1} ^ {n} \eta_j)}
\]
which we recognise as the corresponding Coxeter automorphism associated to our choice of simple roots.
Observe that if $ R_\alpha $ is a root vector for $\alpha =\sum_{j = 1} ^ n m_j\alpha_j $, we have
\[
\sigma (R_\alpha) = e ^ {\frac {\pi i} {n}\sum_{j = 1} ^ mn_j} R_{\alpha}.
\]
For $ j <n $ the simple roots satisfy $\bar\alpha_j = -\alpha_j $ and $\bar\alpha_n = -\alpha_0 $, so the automorphism $\sigma$ preserves the real form $\mathfrak {so} (2n, 1) $ and  the group $ SO (2n, 1) $. Alternatively, this follows directly from  \cite [Proposition 3.1 and Theorem 3.2] {CT:11} which for $\g ^\C$ any simple complex Lie algebra characterised all Cartan involutions of a real form $\g $ such that the corresponding Coxeter automorphism preserves $\g $.

The identity component of the fixed point set of the group automorphism $\sigma $ is $ T $
 and clearly $\sigma  $ has order $ n $, and so it exhibits $ SO (2n, 1)/T $ as an $ n $-symmetric space. 

For any $ k $-symmetric space $ (G/K,\sigma) $ the  induced automorphism $\sigma $ of $\g $ gives a $\Z_k $-grading
\[
\g ^\C =\bigoplus_{j = 0} ^ {k -1}\g ^\sigma_j,\quad [\g ^\sigma_j,\g ^\sigma_l ]\subset\g ^\sigma_{j + l},
\]
 where $\g ^\sigma_j $ denotes the $ e ^ {j\frac {2\pi i} {k} } $-eigenspace of $\sigma $. We have the reductive  splitting
\[
\g =\mathfrak{k}\oplus\p
\]
with
\[
\p ^\C =\bigoplus_{j = 1} ^ {k -1}\g_j ^\sigma,\qquad \mathfrak {k} ^\C =\g_0 ^\sigma,
\]
and if $\varphi $ is a $\g $-valued form we may decompose it as $\varphi =\varphi_{\mathfrak k} +\varphi_{\mathfrak p} $.

A smooth map $ f $ of a surface into a symmetric space $ (G/K,\sigma) $ is harmonic if and only if for some (and hence any) smooth lift $ F: U\rightarrow G $ of $ f: U\rightarrow G/K $, the form $\varphi = F ^{-1}dF $ has the property that for each $\lambda\in S ^ 1 $
\[
\varphi_\lambda =\lambda \varphi '_\p +\varphi_{\mathfrak{k}} +\lambda^ {- 1}\varphi''_\p
\]
satisfies the Maurer-Cartan equation
\[
d\varphi_\lambda +\frac 12 [\varphi_\lambda\wedge\varphi_\lambda ] = 0.
\]

When $ k >2 $, the analogous condition characterises those smooth maps $ \psi$ of a surface  into $ G/K $  such that
for some (and hence any) smooth lift $ F: U\rightarrow G $ of $ \psi: U\rightarrow G/K $,
$\varphi ' = F ^ {- 1}\partial F $ takes values in $\g_0 ^\sigma\oplus\g _1^\sigma $. These maps are termed {\em primitive} and form a subclass of the harmonic maps. (Proofs of these statements are contained for example in \cite {CT:11}.)

The lifts $\psi:\C\rightarrow SO (2n, 1)/T $ that we shall construct for superconformal harmonic $ f:\C\rightarrow S ^ {2n}_1 $ will in fact be cyclic primitive, which is needed to make contact with the affine Toda field equations and  serves to distinguish superconformal maps from the isotropic ones. A smooth map $ \psi:\C\rightarrow SO (2n, 1)/T $ is  \emph{cyclic primitive} if it is primitive and satisfies the  condition that the image of $ F^ {- 1}\partial F $ contains a cyclic element, where as before $F:\C\rightarrow SO (2n, 1) $ is a smooth lift of $\psi $.
Writing $\alpha_0 $ for the lowest root and $  \mathcal G ^\alpha $ for the root space of root $\alpha $, an element in $\left (\bigoplus_{j = 0} ^ n \mathcal G ^{\alpha_j}\right) $ is  \emph{cyclic} if its projection to each of the root spaces $  \mathcal G ^{\alpha_0},  \mathcal G ^{\alpha_1},\ldots ,  \mathcal G ^{\alpha _n} $ is non-zero. 

The observant reader may have noted that we defined primitive earlier only for maps into $ k $-symmetric spaces with $ k >2 $ as for $ k = 2 $ the primitive condition is vacuous. In order that maps into $ S ^ 2_1 $ not be excluded from our consideration, we will take primitive in this case to mean harmonic and cyclic primitive to further mean that the image of $ F^ {- 1}\partial F$ contains a semisimple element.

A map $f:\C\to S ^ {2n}_1 $ has a  primitive lift $\psi:\C\to SO (2n, 1)/T $ if and only if it has a frame $F:\C\rightarrow SO (2n, 1) $ such that $ F ^ {- 1} F_z $ is valued in $\mf{t}^\C\oplus \left (\bigoplus _{j = 0} ^ {n} \mathcal G ^{\alpha_j}\right) $. We call such a frame $ F $ a  \emph{primitive frame} of $ f $. Note that if $ f $ has a primitive lift $\psi $ then all frames $ F $ of $\psi $ will be primitive frames of $ f $. We now explicitly characterise the primitive frames corresponding to a natural choice of simple roots for $\mathfrak {so} (2n, 1,\C) $.

A root space decomposition of $ SO (2n, 1) $ is given in the table below, where throughout the table we assume $1\leq j<k<n$.
\begin{center}
Root  space  decomposition of $SO(2n,1)$
\begin{tabular}{c||c}
\hline
$\alpha \in \mathfrak{t}^*$ & $R_\alpha \in \mathfrak{g}/\mathfrak{t}$\\
\hline
$i \tilde {a}_k$ & $[1,2k]+i[1,2k+1]$\\
$-i\tilde {a}_k$ & $[1,2k]-i[1,2k+1]$\\
$\tilde {a}_n$ & $[1,2n]+[1,2n+1]$\\
$-\tilde {a}_n$ & $[1,2n]-[1,2n+1]$\\
$i\tilde {a}_j+i\tilde {a}_k$ & $[2j,2k] + i[2j,2k+1] +i[2j+1,2k] -[2j+1,2k+1]$\\
$i\tilde {a}_j-i\tilde {a}_k$ & $[2j,2k] - i[2j,2k+1] +i[2j+1,2k] +[2j+1,2k+1]$\\
$-i\tilde {a}_j+i\tilde {a}_k$ & $[2j,2k] + i[2j,2k+1] -i[2j+1,2k] +[2j+1,2k+1]$\\
$-i\tilde {a}_j-i\tilde {a}_k$ & $[2j,2k] - i[2j,2k+1] -i[2j+1,2k] -[2j+1,2k+1]$\\

$i\tilde {a}_j+\tilde {a}_n$ & $[2j,2n] - [2j,2n+1] +i[2j+1,2n] -i[2j+1,2n+1]$\\

$i\tilde {a}_j-\tilde {a}_n$ & $[2j,2n] + [2j,2n+1] +i[2j+1,2n] +i[2j+1,2n+1]$\\
$-i\tilde {a}_j+\tilde {a}_n$ & $[2j,2n] - [2j,2n+1] -i[2j+1,2n] +i[2j+1,2n+1]$\\
$-i\tilde {a}_j-\tilde {a}_n$ & $[2j,2n] + [2j,2n+1] -i[2j+1,2n] -i[2j+1,2n+1]$\\
\hline
\end{tabular} \end{center}

Primitive frames $ F $ are those for which $ F ^ {- 1} F_z $ is of the form
\begin{equation}\label{eq:matrix}
\tiny
{\left(
\begin{array}{c|cc|cc|cc|cc|cc}
 & -c_1 & -ic_1 & \ & \ & \ &\ & \ & \ \\
\hline
c_1  & 0    &    a_1 & -c_2 & -ic_2   &    &  &  & \\
ic_1  & - a_1    &    0 & ic_2 & -c_2   &    &   &  &\\
\hline
& c_2& -ic_2 & 0 & a_2 & -c_3& -ic_3 & & & & \\
& ic_2&c_2 & - a_2 & 0 & ic_3& -c_3 & & & & \\
\hline
  & & & c_3  & -ic_3  & 0  & a_3 & \ddots  &   &   \\
 & &  & ic_3  &  c_3 & - a_3  & 0 &  &  \ddots & \\
\hline
 & & & & & \ddots &  & \ddots  &   & -c_n -c_0  & -c_n + c_0   \\
 & &  & & & &\ddots   &   &\ddots &  ic_n +ic_0 & ic_n -ic_0\\
\hline
  &  &  & & & & &  c_n + c_0 &  -ic_n -ic_0  &  0  &    a_n\\
  &  & & & & &  &  -c_n + c_0 & ic_n -ic_0  &    a_n & 0
\end{array}
\right)}.
\end{equation}

\begin{theorem}\label {thm:cyclicprimitive}
A harmonic map $f:\C \to S^{2n}_1$  has a cyclic primitive lift $\psi:\C\rightarrow SO (2n, 1)/T $ if and only if it is superconformal and the entries $\{f_1, \ldots, f_{n}\}$ of its harmonic sequence are defined everywhere. 
 Furthermore, under these conditions, for each  $ 1\leq j\leq n -1 $ the function $ f_j $ is spacelike away from its (discrete) vanishing set and we have
$$f_j= 2^{j-1}c_1\ldots c_j F(e_{2j} + ie_{2j+1})\text { for each } 1\leq j\leq n-1
$$
and
$$\langle \partial_z^n f, \partial_z^n f\rangle=2^{2n-2}c_1^2c_2^2 \ldots c_{n-1}^2 c_n c_0 $$
for any cyclic primitive frame $F:\C \to SO(2n,1)$ of $f$. Here the functions $ c_j $ are the root vector coefficients defined in  \eqref {eq:matrix}, that is
the $\g_1 ^\sigma $-component of $ F^{-1}F_z $ is $ \sum_{k = 0} ^ n c_kR_{\alpha_k} $.
\end{theorem}

\begin{proof}
First suppose that $f$ has a cyclic primitive frame  $F :\C \to SO(2n,1)$. Set
\begin {equation}\label {eq:primitivesequence}
 f_0=f\text { and } f_j:=2^{j-1}c_1\ldots c_j F(e_{2j} + ie_{2j+1})
\end {equation}
 for each $1\leq j\leq n-1$. We will prove $f$ is superconformal and that these $f_j$ agree with the harmonic sequence of $ f $.
This latter statement comes from showing that the $f_j$ satisfy \eqref{harm}, which we prove by induction. The $j=1$ case follows from $\langle \partial_z f, f\rangle =0$ and
\begin{align*}
\langle \partial_z f, F e_k\rangle &=
\begin{cases}
c_1 & \text{ if } k=2\\
i c_1 & \text{ if } k=3\\
0 & \text{otherwise}.
\end{cases}
\end{align*}
For the inductive step consider $1\leq j \leq n-2$. From
\[
f_j:=2^{j-1}c_1\ldots c_j F(e_{2j} + ie_{2j+1})
\]
 we have
\[
\partial_z f_j= \partial_z(2^{j-1}c_1\ldots c_j ) F(e_{2j} + ie_{2j+1}) + 2^{j-1}c_1\ldots c_j F_z(e_{2j} + ie_{2j+1})
\]
 and hence
$$\partial_z f_j - \frac{\langle \partial_z f_j, \bar{f_j}\rangle}{\langle f_j, \bar{f_j}\rangle}f_j =
2^{j-1}c_1\ldots c_j (F_z(e_{2j} + ie_{2j+1}) -  i a_j F(e_{2j}+ie_{2j+1})), $$
where $ a_j $ is also defined in  \eqref {eq:matrix}.
This implies
\begin{align*}
\langle \partial_z f_j - \frac{\langle \partial_z f_j, \bar{f_j}\rangle}{\langle f_j, \bar{f_j}\rangle}f_j, F e_k\rangle
&=
\begin{cases}
2^j c_1 \ldots c_jc_{j+1} & \text{ if } k=2j+2\\
i2^j c_1 \ldots c_jc_{j+1}  & \text{ if } k=2j+3\\
0 & \text{otherwise}
\end{cases}
\end{align*}
which proves \eqref {eq:primitivesequence} for $ j +1 $. 
 Thus we know that the sequence $\{f_0, f_1, \ldots f_{n-1}\}$ agrees with the harmonic sequence for $f$.

Importantly this characterisation of the $f_j$ tells us inductively that
\[
f_j = \partial_z^j f + \sum_{k=1}^{j-1} u_k \partial_z^k f\text { for some functions } u_k.
\]
 From $F \in SO(2n,1)$ and our choice of $f_j$, we have that $\langle f_j, f_k \rangle =0$ for $1\leq j\leq n-1$ and $0\leq k\leq n-1$. Induction then shows that $\langle \partial_z^\alpha f, \partial_z^\beta f\rangle = 0$ for $1\leq \alpha +\beta \leq 2n-2$ and $\alpha, \beta \geq 0$. From $\langle \partial_z^{n-1} f, \partial_z^{n-1} f\rangle = 0$ we also obtain $\langle \partial_z^{n} f, \partial_z^{n-1} f\rangle = 0$ by taking the $z$-derivative.

To prove that $f$ is superconformal it now only remains to show that $\langle \partial_z^n f, \partial_z^n f\rangle$ is not the zero function. But
\begin{align*}
\langle \partial_z^n f, \partial_z^n f\rangle&= \langle \partial_z f_{n-1},  F e_{2n}\rangle^2 -  \langle \partial_z f_{n-1},  F e_{2n+1}\rangle^2\\
&=2^{2n-4}c_1^2 \ldots c_{n-1}^2\langle  F^{-1}F_z(e_{2n-2}+i e_{2n-1}), e_{2n}\rangle^2\\
&\qquad -2^{2n-4}c_1^2 \ldots c_{n-1}^2\langle F^{-1}F_z(e_{2n-2}+i e_{2n-1}), e_{2n+1}\rangle^2\\
&=2^{2n-2}c_1^2\ldots c_{n-1}^2c_nc_0
\end{align*}
and $F$ is a cyclic primitive frame, so these $c_k$ may vanish only isolated points.  Hence $\langle \partial_z^n f, \partial_z^n f\rangle$ is also nonzero except possibly on a discrete set.

It follows directly from  \eqref {eq:primitivesequence} that for $ 1\leq j\leq n -1 $, the functions $ f_j $ are space-like at all points at which they do not vanish. At any point $ z\in\C $ such that either $|| f_j (z)||\neq 0 $ we may use  \eqref {harm} to define $ f_{j +1} (z) $ and if $ f_j (z) = 0 $ then this definition extends by continuity. So the fact that the only points where $ f_{n -1} $ has zero norm are the points where it vanishes guarantees that $ f_{n} $ is defined everywhere.

To complete the proof it only remains to provide a cyclic primitive frame $F:\C \to SO(2n,1)$ of $f:\C \to S^{2n}_1$ when $f$ is a superconformal harmonic map such that the elements $\{ f_1 \ldots f_{n}\}$ of its harmonic sequence are defined everywhere.
 In particular then for $ 1\leq j\leq n -1 $, the functions $ f_j $ have zero norm only at points where they vanish. Away from the  zero set of $ f_j $, we may define all but the last two columns of $ F  $ by
\begin{equation}\label{eq:realframe}
Fe_1= f,\quad Fe_{2j} = \frac{f_j + \overline{f_j}}{\sqrt{2}\|f_j\|}, \quad Fe_{2j+1} =i\frac{\overline{f_j} - f_j}{\sqrt{2}\|f_j\|} \text { for $1\leq j \leq n-1$}
\end{equation}
and use continuity to extend this definition to the isolated points at which some $ f_j $ vanishes. The orthogonal space to the first $2n-1$ columns of $F$ (ie the space orthogonal to $\text{span }\{f, f_1, \bar{f_1}, \ldots, f_{n-1}, \overline{f_{n-1}}\}$) is the complexification of a real two-dimensional space with signature (1,1). Thus we can find  smoothly changing $v,w$ in this space such that $\langle v,v\rangle=1 = -\langle w,w\rangle$. Define $Fe_{2n}$ and $Fe_{2n+1}$ to be $v$ and $w$ respectively.

Note that from $\langle f_j, f_k \rangle =0$ and $\langle f_j, \overline{f_k}\rangle = 0$ we have that $F$ is valued in $SO(2n,1)$. Verifying that $F^{-1}F_z \in \g ^\sigma_0  \oplus \g ^\sigma_1$ may be achieved by  a direct calculation in which one rewrites matrix entries as bilinear products of the $f_j, \bar{f_j}, \partial_z f_j, \partial_z\bar{f_j}, Fe_{2n}$, and $Fe_{2n+1}$ and then employs the following observations 
\begin{align*}
\langle f_j, \partial_z f_k \rangle &= 0 = \langle \bar{f_j}, \partial_z \bar{f_k}\rangle & \text{ for all } j,k\\
\langle f_j, \partial_z \bar{f_k} \rangle& = 0 = \langle \bar{f_j}, \partial_z  f_k\rangle & \text{ for } j \neq k+1\\
\langle Fe_{k}, f_j \rangle &= 0 = \langle Fe_{k}, \bar{f_j}\rangle = \langle Fe_k, \partial_z \bar{f_j}\rangle & \text{ for all $j$ and $k=2n,2n+1$}\\
\langle Fe_k, \partial_z f_j \rangle &=0 & \text{ for $j\neq n-1$ and $k=2n,2n+1$}.
\end{align*}
\end{proof}

\begin{rem}
If we omit the cyclic condition, we find that harmonic $f$ with primitive lifts correspond to either isotropic or superconformal harmonic maps such that the entries $\{f_1, f_2,\ldots , f_n\} $ of the harmonic sequence of $ f $ are defined everywhere.
\end{rem}

\begin{cor}\label{cor:primitive}
Let $f:\C \to S^{2n}_1$ be a doubly periodic superconformal map such that the elements $ f_1, \ldots, f_n$ of the harmonic sequence of $ f $ are everywhere defined. Then $\langle \partial_z^n f, \partial_z^n f\rangle$ is a non-zero constant and $ f_1 \ldots f_{n-1}$ are all spacelike.
\end{cor}
\begin{proof}
The double-periodicity of $f$ implies that $\langle \partial_z^n f, \partial_z^n f\rangle$ is also doubly periodic. Since $ f $ is harmonic we know that $\partial_z \partial_{\bar {z}} f\in \C f$ and hence
\[
\partial_{\bar {z}} \langle \partial_z^n f, \partial_z^n f\rangle = 2\langle \partial_z^{n-1}(\partial_z \partial_{\bar {z}} f), \partial_z^n f \rangle =0.
\]
 Thus $\langle \partial_z^n f, \partial_z^n f\rangle$ must be constant and as $f$ is superconformal this constant is non-zero. The formula for $ \langle \partial_z^n f, \partial_z^n f\rangle$ in Theorem~\ref{thm:cyclicprimitive} then implies that the root vector coefficients $c_j $ in  \eqref {eq:matrix} never vanish. Hence for each $ j = 1, \ldots , n  -1 $ the function $f_j = 2^{j-1}c_1 \ldots c_jF(e_{2j}+ie_{2j+1})$ is spacelike everywhere.
\end{proof}

%

\section{Toda frame and Finite type result}\label{Toda}
For the chosen simple roots  $\alpha_1,\ldots,\alpha_n $ and lowest root $\alpha_0 $ of $\mathfrak {so} (2n, 1) $ (see page ~\pageref {page:simple}) 
the permutation $\pi $ of the roots defined by
\[
\pi (\alpha_j ) = \alpha_j, \; 1 <j <n \text { and }\pi (\alpha_n ) = \alpha_0 
\]
satisfies
\[
\overline { \alpha_j} = -\alpha_{\pi (j)}.
\]
where $\Omega:\C\rightarrow i\t $ and the root vectors $ R_{\alpha_j} $ and $ m_j\in\R ^ + $

For any simple real Lie algebra $\g $ with simple roots $\alpha_1, \ldots ,\alpha_N $ and lowest root $\alpha_0 $, in  \cite{CT:11} we consider the slight generalisation of the two-dimensional affine Toda equation
\begin{equation}\label{eq:Toda}
2\Omega_{z\bz}  = \sum_{j = 0}^ N m_j e^{2 \alpha_j (\Omega)} [R_{\alpha_j}, R_{-\alpha_j}]
\end{equation} 
where $\Omega:\C\rightarrow i\t $ and the root vectors $ R_{\alpha_j} $ and  positive real numbers $ m_j$
satisfy the reality conditions
\[
\overline {m_j} = m_{\pi (j)},\quad\overline {R_{\alpha_j}} = R_{-\pi (\alpha_j)}.
\]

The classical case of the affine Toda equation is the one-dimensional equation for $\g = \mathfrak {su} (n) $, in which $\Omega_{z\bar z} $ is replaced by the second derivative of $\Omega $ with respect to a single real variable, the root vectors
are chosen so that $ [R_{\alpha_j}, R_ {-\alpha_j}] $ is the dual of $\alpha_j $ with respect to the Killing form and
\[
\alpha_0 = -\sum_{j = 1} ^ Nm_j\alpha_j.
\]
 This describes the motion of particles arranged in a circle, connected by springs with exponential potentials.

Given a cyclic element $W = \sum_{j = 0}^N r_j R_{\alpha_j}$ of
$\mathfrak{g}^{\sigma}_1$ with  $ r_{\pi (j)} =\overline {r_j} $ and $\overline {R_{\alpha_j}} = R_{-\alpha_{\pi (j)}}$, we say that a lift $F : \C
\to G$ of $\psi : \C \rightarrow G / T$ is a \emph{Toda frame} with
respect to $W$ if there exists a smooth map $\Omega : \C \rightarrow i
\mathfrak{t}$ such that
\begin{equation}
  \label{eq:Todaframe} F^{- 1} F_z = \Omega_z + \Ad_{\exp \Omega} W.
\end{equation}

We call $\Omega$ an \emph{affine Toda field} with respect to $W$. Then
\cite [Lemma 4.1] {CT:11}
the affine Toda field equation  \eqref {eq:Toda}  is the integrability condition for the existence
  of a Toda frame with respect to $ W $ where $m_j = r_j \overline{r_j}$ for $j
  = 0, \ldots, N$.
\begin{theorem}\cite [Theorem 4.2] {CT:11} \label {theorem:Toda}
A map $\psi:\C\rightarrow G/T $ possesses a Toda frame if and only if it has a cyclic primitive frame for which $c_0 \prod_{j =
  1}^ N c_j^{m_j}$ is constant.

More precisely, let $\psi : \C \to G / T$ be a cyclic primitive map possessing a frame $\tilde{F} : \C \to G$  such that $c_0 \prod_{j =
  1}^ N c_j^{m_j}$ is a non-zero constant, where $ c_j $ are the coefficients with respect to any fixed choice of root vectors.\footnote{Observe that
whether $c_0 \prod_{j =
  1}^ N c_j^{m_j}$ is  a non-zero constant is independent of the choice of root vectors.}  Then for any cyclic element $W $ of $\g ^\sigma_1 $ which is normalised with respect to $\tilde F $
  there exists a Toda frame $F : \C \to G$ of $\psi$ with respect to $W$. Furthermore if $\psi$ and $\tilde F$ are doubly periodic with lattice $\Lambda$  then so is the Toda frame $ F $.

Conversely, if $\psi:\C \to G/T$ has a Toda frame $ F $ with respect to cyclic $ W\in\g_1 ^\sigma $
then $\psi$ is cyclic primitive and $ W $ is normalised with respect to $ F $. In particular then the root coefficients $c_j$ are such that $c_0 \prod_{j =
  1}^ N c_j^{m_j} $ is constant. 
  
\end{theorem}

When combined with Theorem~\ref{thm:cyclicprimitive} this enables us to give the following relationship between superconformal harmonic maps into de Sitter spaces and   Toda frames.
\begin {corollary}\label {corollary:Toda}
Let  $ f:\C\rightarrow S ^ {2n}_1 $ be a superconformal harmonic map such that the elements $ f_1, \ldots, f_n$ of the harmonic sequence of $ f $ are everywhere defined and the root space coefficients $ c_j $ are such that $c_0 c_1 ^ 2c_2 ^ 2 \ldots c_{n -1} ^ 2c_n $ is constant. Then $ f $ has a lift $\psi:\C\rightarrow SO (2n, 1)/T $ possessing a Toda frame with respect to any appropriately normalised cyclic element $ W $.

 If $ f $ is doubly-periodic with respect to some lattice $\Lambda\subset\C $ then the condition on the root space coefficients is automatically satisfied and furthermore the lift $\psi $ is periodic with respect to $\Lambda $.
\end {corollary}
\begin {proof} The first and last statements are immediate from Theorems ~\ref{thm:cyclicprimitive} and \ref {theorem:Toda}. It remains only to show that  $c_0 c_1 ^ 2c_2 ^ 2 \ldots c_{n -1} ^ 2c_n $  is constant. But it was shown in the proof of Theorem~\ref{thm:cyclicprimitive},  that 
\[
c_0 c_1 ^ 2c_2 ^ 2 \ldots c_{n -1} ^ 2c_n = 2 ^ {2 - 2n} \langle \partial ^ n_z f,\partial ^ n _{ z} f \rangle 
\]
and hence this quantity is holomorphic for harmonic $ f $. Here the $ c_j $ are as defined in  \eqref {eq:matrix}, that is they are the root space coefficients for the root space decomposition given above  \eqref {eq:matrix}.\end {proof}

The existence of a Toda frame has an important consequence, as we recall from\cite {CT:11}:
\begin{theorem}\cite[Theorem 5.2] {CT:11}
\label{thm:finite}
 Let $ G $ be a simple real Lie group and $ T $ a Cartan subgroup such that $G/T$ is preserved by a Coxeter automorphism $\sigma $.
Suppose $\psi: \Tor\rightarrow G/T$ has a Toda frame $ F:\Tor \to G$. Then $\psi $ is of finite type.
\end{theorem}

From Theorem~\ref {thm:finite} and Corollary~\ref{corollary:Toda} we  reach our main conclusion: 
\begin {corollary}\label {cor:finite}
Let $ f:\Tor\rightarrow S ^ {2n}_1 $ be a superconformal harmonic map with globally defined harmonic sequence $\{f_1, \ldots , f_n\} $. Then $ f $ has a lift $\psi:\Tor\rightarrow SO (2n, 1)/T $ of finite type.
\end {corollary}

 \section {Applications to Willmore  surfaces  in $ S ^ 3 $}\label{willmore}
In this section we explain the implications of our results to Willmore tori in $ S ^ 3 $.
Recall that an immersed surface $\upsilon:\Sigma\rightarrow S ^ 3 $ is \emph{Willmore} if it is critical for the Willmore functional
\[
\mathcal W =\int_{\Sigma} H ^ 2 \; dA,
\]
where  $ H $ denotes the mean curvature of $\upsilon $ and $ dA $ the area form. The relationship between these surfaces and Corollary~\ref{cor:finite} is provided by the conformal Gauss map, which allows us to study Willmore surfaces in $S^3$ without umbilic points in terms of minimal surfaces in $S^4_1$. The conformal Gauss map $f $  of a Willmore surface is either isotropic or superconformal, and in \cite{Bryant:84} Bryant constructed all isotropic examples using a Weierstrass-type representation. Hence the remaining case is when $f $ is superconformal, and we shall see that all Willmore tori in $ S ^ 3 $ without umbilic points and with superconformal Gauss maps are of finite type.

Geometrically, the conformal Gauss map of an immersion $\upsilon:\Sigma\rightarrow S ^ 3 $ associates to each point on the surface $\Sigma $ its central sphere, that is the oriented 2-sphere in $ S ^ 3 $ with the same normal vector and mean curvature. 
A sphere in $S^3$ is the intersection of $S^3$ and a hyperplane in $\mathbb{R}^4$;
\[
S^3 \cap \{x_1,x_2,x_3,x_4 : a_1x_1+a_2x_2+a_3x_3+a_4x_4-b=0\}.
\]
For this hyperplane to intersect with $S^3$ at more than one point requires $a_1^2+a_2^2+a_3^2+a_4^2-b^2>0$ and hence we can scale $(a_1, a_2, a_3, a_4, b)$ so that $a_1^2+a_2^2+a_3^2+a_4^2-b^2=1$. Thus each sphere can be identified with two antipodal points $\pm (a_1, a_2, a_3, a_4, b)\in S ^ 4_1 $; choosing an orientation for the sphere gives a well-defined element of $ S ^ 4_1 $ and so we see that the   space of oriented 2-spheres in $S^3$ is naturally identified with $S^4_1$. The conformal Gauss map $f:\Sigma\rightarrow S ^ 4_1 $ is given explicitly by
\[
f (z) =H (z)\cdot \Upsilon (z) + N (z)
\]
where $\Upsilon (z) = (\upsilon (z), 1)$, $ N = (n,0)$.
As observed in \cite{Bryant:84} it is weakly conformal and an immersion away from the umbilic points of $\upsilon $. A simple computation shows that the area form on $\Sigma $ induced by $f $ is given by $(H ^ 2 - K) dA $, where $ K $ is the Gaussian curvature of $\upsilon $. Using the Gauss-Bonnet theorem, we may replace the Willmore functional by the conformally invariant functional $\int_{\Sigma} (H ^ 2 - K) dA $. 
We see then that $\upsilon:\Sigma\rightarrow S ^ 3 $ is a Willmore immersion without umbilic points if and only if $f:\Sigma\rightarrow S ^ 4_1 $ is a minimal immersion, or equivalently is conformal and harmonic.

Let $f:\C\rightarrow S ^ 4_1 $ be a minimal immersion (not necessarily superconformal).
We call a frame $ F:\C\rightarrow SO (4, 1) $ for $f $ \emph{adapted} if $ Fe_1 =f $ whilst $ Fe_2, Fe_3 $ span a tangent plane to $f $. From Theorem~\ref{thm:cyclicprimitive} it easily follows that
\begin {corollary}
An immersion $f:\C\rightarrow S ^ 4_1 $ is conformal and harmonic if and only if it possesses an adapted primitive frame $ F $.
\end {corollary}

\begin{proof}
The forwards implication is contained in the proof of Theorem~\ref{thm:cyclicprimitive}. Suppose conversely that $f $ has a primitive adapted frame $ F $. Then $ f $ is harmonic because it is the projection of a primitive map. It is conformal because having a primitive frame forces it to be either isotropic or superconformal. Alternatively, from \eqref {eq:matrix} we see that
$ f_z = F^ {- 1} F_ze_1\in\C (e_2 - ie_3) $
and hence see directly that $f $ is conformal.
\end {proof}
%
Observe that the frame $ F $ constructed here is the same as that described in section ~\ref{primitive}, but since no reference to the harmonic sequence was needed, we have $ f_1 = f_z $ and then last two columns of $ F $ are chosen so that $ F $ lies in $ SO (4, 1) $.

A straightforward computation shows that for any minimal $ f $, the quantity $ \langle f_{zz}, f_{zz} \rangle $ is holomorphic. Hence if we further assume that $ f $ is doubly periodic this quantity is constant. The curvature ellipse of $ f $ at $ z\in T $ is the image of the unit circle under the second fundamental form and is a circle precisely when $ \langle f_{zz}, f_{zz} \rangle = 0 $ which occurs if and only if $ f $ is isotropic. Alternatively the first ellipse of curvature being a non-circular ellipse corresponds  to $ f $ being superconformal, and the last two columns of $ F $ are then determined by the principal directions of this ellipse. A minimal $ f: T ^ 2\rightarrow S ^ 4_1 $ is thus either isotropic or superconformal.

From  the expression
\[
\langle f_{zz}, f_{zz}  \rangle = 4c_1 ^ 2c_2c_0
\]
obtained in Theorem~\ref{thm:cyclicprimitive} we see that the frame $ F $ constructed above is cyclic precisely when $ f $ is said to conformal.

Corollary~\ref{cor:finite}, applied to the case $ n = 2 $, now shows that
\begin {corollary}
If $\upsilon: T ^ 2\rightarrow S ^ 3 $ is a non-isotropic Willmore immersion without umbilic points then its conformal Gauss map has a primitive lift of finite type.
\end {corollary}
 Hence we may recover $\upsilon $ by integrating a pair of commuting vector fields on a finite dimensional Lie algebra. \bibliographystyle{plain}

\def\cprime{$'$}

\end {document}